\documentclass[10pt,reqno]{amsart}
\usepackage{amsmath}
\usepackage{amsthm}
\usepackage{amssymb}
\usepackage{amsfonts}
\usepackage{latexsym}
\usepackage{hyperref}
\usepackage{cite}
\newcommand{\C}{ \mathbb{C}}
\newcommand{\D}{ \mathbb{D}}
\newcommand{\dD}{ \partial\mathbb{D}}

\newcommand{\R}{ \mathbb{R}}

\renewcommand{\Re}{\operatorname{Re}}

\newcommand{\diag}{\operatorname{diag}}
\newcommand{\tr}{\operatorname{tr}}
\newcommand{\norm}[1]{\| #1 \|}
\newcommand{\inner}[1]{\langle #1 \rangle}

\newcommand{\K}{\mathcal{K}}
\newcommand{\minimatrix}[4]{\begin{pmatrix} #1 & #2 \\ #3 & #4 \end{pmatrix}  }
\newcommand{\megamatrix}[9]{\begin{pmatrix} #1 & #2 & #3 \\ #4 & #5 & #6 \\ #7 & #8 & #9\end{pmatrix}  }

\renewcommand{\vec}[1]{{\bf #1}}

\renewcommand{\phi}{\varphi}

\theoremstyle{plain}
\newtheorem{Theorem}{Theorem}

\newtheorem{Corollary}[Theorem]{Corollary}
\newtheorem{Lemma}[Theorem]{Lemma}
\theoremstyle{definition}

\newtheorem{Example}{Example}

\newtheorem{Question}{Question}

\numberwithin{Theorem}{section}
\allowdisplaybreaks

\begin{document}
%\linenumbers
\bibliographystyle{plain}

    \title[Unitary equivalence to a truncated Toeplitz operator]{Unitary equivalence to a truncated Toeplitz operator:  analytic symbols}

    \author{Stephan Ramon Garcia}
    \address{   Department of Mathematics\\
            Pomona College\\
            Claremont, California\\
            91711 \\ USA}
    \email{Stephan.Garcia@pomona.edu}
    \urladdr{http://pages.pomona.edu/\textasciitilde sg064747}

\author{Daniel E.~Poore}
	\email{dep02007@mymail.pomona.edu}

	\author{William T. Ross}
	\address{   Department of Mathematics and Computer Science\\
	            University of Richmond\\
	            Richmond, Virginia\\
	            23173 \\ USA}
	\email{wross@richmond.edu}
	\urladdr{http://facultystaff.richmond.edu/~wross}


    \thanks{First author partially supported by National Science Foundation Grant DMS-1001614.}

    \begin{abstract}
	    Unlike Toeplitz operators on $H^2$, truncated Toeplitz operators do not have a natural
	    matricial characterization.  Consequently, these operators are difficult to study
	    numerically.  In this note we provide criteria for a matrix with distinct eigenvalues to be
	    unitarily equivalent to a truncated Toeplitz operator having an analytic symbol.  This test is constructive
	    and we illustrate it with several examples.  As a byproduct, we also prove that 
	    every complex symmetric operator on a Hilbert space of dimension $\leq 3$ is unitarily equivalent to a direct
	    sum of truncated Toeplitz operators.
    \end{abstract}

\maketitle

\section{Introduction}
	Interest in truncated Toeplitz operators has blossomed over the last several years
	\cite{BCFMT, MR2597679, TTOSIUES, CRW, NLEPHS, Sarason, MR2468883, MR2418122, Sed, STZ},
	sparked by a seminal paper of D.~Sarason \cite{Sarason}.
	Unfortunately, only the very simplest truncated Toeplitz operators, the finite Toeplitz matrices, have
	any sort of practical matricial description.  Thus the numerical study of truncated Toeplitz operators
	is difficult, even in low dimensions.
	For instance, while the theory of pseudospectra for Toeplitz matrices has recently undergone
	spirited development \cite[Ch.~II]{Trefethen}, many basic questions
	about truncated Toeplitz operators remain unanswered.
	For instance, finding a characterization of rank-two self-adjoint
	truncated Toeplitz operators is still an open problem \cite[p.~508]{Sarason}, despite
	the fact that all rank-one truncated Toeplitz operators have already been identified \cite[Thm.~5.1]{Sarason}.
	
	Although a few results concerning matrix representations of
	truncated Toeplitz operators have been obtained \cite{CRW, TTOSIUES, STZ}, the general question of 
	determining whether a given matrix represents a truncated Toeplitz operator, with respect to some orthonormal basis,
	appears difficult.   On the other hand, it is known that every truncated Toeplitz operator
	is unitarily equivalent to a complex symmetric matrix \cite{CSOA,CCO}, a somewhat more general issue
	which has been studied by several authors \cite{UECSMGC, UECSMMC, Tener, Vermeer}.  
	
	Our work is partly motivated by the question of whether truncated Toeplitz operators serve as some sort of 
	model for complex symmetric operators.  A significant amount of evidence has been produced in this 
	direction, starting with D.~Sarason's early observation that the Volterra integration operator is unitarily equivalent
	to a truncated Toeplitz operator \cite{MR0192355} (see also \cite[p.~41]{MR1892647}).  Since then,
	many other examples of complex symmetric operators which are representable in terms of 
	truncated Toeplitz operators have emerged.  For instance,
	normal operators \cite[Thm.~5.6]{TTOSIUES}, rank-one operators, \cite[Thm.~5.1]{TTOSIUES}, 
	$2 \times 2$ matrices \cite[Thm.~5.2]{TTOSIUES}, and inflations of finite Toeplitz matrices 
	\cite[Thm.~5.7]{TTOSIUES} are unitarily equivalent to truncated Toeplitz operators. The recent
	preprint \cite{STZ} contains a host of other examples.
	
	Before stating our main results, we briefly review some of the necessary preliminaries.
	Let $H^2$ denote the Hardy space on the open unit disk $\D$,
	let $H^{\infty}$ denote the Banach algebra of all bounded analytic functions on $\D$,
	and let $L^{\infty}:= L^{\infty}(\dD)$ and $L^2:= L^2(\dD)$ denote the
	usual Lebesgue spaces on the unit circle $\dD$ \cite{Duren, MR2261424}.
	To each nonconstant inner function $\Theta$ we associate the \emph{model space} 
	$\K_{\Theta} := H^2 \ominus \Theta H^2$,
	which is the reproducing kernel Hilbert space corresponding to the kernel
	\begin{equation*}
		K_{\lambda}(z)
		:= \frac{1 - \overline{\Theta(\lambda)} \Theta(z)}{1 - \overline{\lambda} z}, \quad z,\lambda \in \D.
	\end{equation*}
	For our purposes, we find it more convenient to work with the \emph{normalized kernels} 
	$k_{\lambda} := \sqrt{\frac{1-|\lambda|^2}{1 - |\Theta(\lambda)|^2} } K_{\lambda}$. 
	The space $\K_{\Theta}$ carries a natural \emph{conjugation} (an isometric, conjugate-linear, involution)
	\begin{equation}\label{eq-ModelConjugation}
		C f := \overline{ f z} \Theta,
	\end{equation}
	defined in terms of boundary functions \cite{CCO, CSOA, NLEPHS}.  The normalized \emph{conjugate kernels}
	$\tilde{k}_{\lambda} := Ck_{\lambda}$ are of particular interest to us.

	For each \emph{symbol} $\phi$ in $L^2$ the corresponding \emph{truncated Toeplitz operator}
	$A_{\phi}$ is the densely defined operator on $\K_{\Theta}$ given by
	\begin{equation} \label{eq-TTODefinition}
		A_{\phi} f := P_{\Theta}(\phi f).
	\end{equation}
	When we wish to be specific about the inner function $\Theta$, we often write
	$A^{\Theta}_{\phi}$ in place of $A_{\phi}$.
	The adjoint of $A_{\phi}$ is the operator $A_{\overline{\phi}}$ and it 
	is easy to see that $A_{\phi} = CA_{\phi}^* C$ where $C$ denotes the conjugation \eqref{eq-ModelConjugation}.
	In other words, $A_{\phi}$ is a \emph{complex symmetric operator} \cite{CCO,CSOA, G-P-II} and hence it can be represented as a complex 
	symmetric (i.e., self-transpose) matrix with respect to some orthonormal basis of $\K_{\Theta}$ \cite{CCO}
	(see also \cite[Sect.~5.2]{NLEPHS}).  	
	
	We are most interested in the case where $\phi \in H^{\infty}$ and $\Theta$ is a finite Blaschke product
	having distinct zeros $z_1,z_2,\ldots,z_n$.  In this case we have
	\begin{equation*}
		k_{z_i} = \frac{\sqrt{1-|z_i|^2} }{1 - \overline{z_i}z}, \qquad
		\tilde{k}_{z_i} = \frac{\sqrt{1-|z_i|^2}\Theta(z) }{z - z_i}.
	\end{equation*}
	For each $\phi$ in $H^{\infty}$ the
	eigenvalues of the \emph{analytic} truncated Toeplitz operator 
	$A_{\phi}^{\Theta}$ are given by $\phi(z_1), \phi(z_2),\ldots, \phi(z_n)$,
	with corresponding normalized eigenvectors $\tilde{k}_{z_1}, \tilde{k}_{z_2}, \ldots, \tilde{k}_{z_n}$ \cite[p.~10]{AMC}.
	In particular, nonzero eigenvectors of an analytic truncated Toeplitz operator are never orthogonal to each other.
	
	One of the main results of this note is the following simple criterion  for determining
	whether or not a given matrix is unitarily equivalent to a trunctated
	Toeplitz operator having an analytic symbol.  

	\begin{Theorem}\label{TheoremMain}
		If $M \in {\bf M}_n(\C)$ has distinct eigenvalues $\lambda_1, \lambda_2, \ldots, \lambda_n$ with corresponding
		unit eigenvectors $\vec{x}_1,\vec{x}_2,\ldots,\vec{x}_n$, then
		$M$ is unitarily equivalent to an analytic truncated Toeplitz operator 
		if and only if there exist distinct points $z_1,z_2,\ldots,z_{n-1}$ in $\D$ such that 
		\begin{equation}\label{eq-TripleExplicit}
			\inner{\vec{x}_n,\vec{x}_i}\inner{\vec{x}_i,\vec{x}_j} \inner{\vec{x}_j,\vec{x}_n} 
			= \frac{ (1 - |z_i|^2)(1-|z_j|^2) }{ 1 - \overline{z_j}z_i}
		\end{equation}
		holds for $1 \leq i \leq j <n$.
	\end{Theorem}
	
	The method of Theorem \ref{TheoremMain} is constructive, in the sense that if \eqref{eq-TripleExplicit}
	is satisfied, then we can construct an inner function $\Theta$ and a polynomial $\phi$ such that 
	$M$ is unitarily equivalent to $A_{\phi}^{\Theta}$ (denoted $M \cong A_{\phi}^{\Theta}$).
	In fact, $\Theta$ is the Blaschke product having zeros at $z_1,z_2,\ldots,z_{n-1}$ and $z_n = 0$.
	
	Although the condition \eqref{eq-TripleExplicit} appears
	somewhat complicated, it encodes a wealth of geometric data.  For instance,
	setting $j=i$ in \eqref{eq-TripleExplicit} yields $|\inner{ \vec{x}_i, \vec{x}_n}|^2 = 1 - |z_i|^2$.
	This in turn provides us with the formula
	\begin{equation}\label{eq-Moduli}
		|z_i| = \sqrt{1 - | \inner{ \vec{x}_i, \vec{x}_n}|^2}
	\end{equation}
	for the moduli of the unknown points $z_1,z_2,\ldots,z_{n-1}$.  Furthermore,
	the proof of Theorem \ref{TheoremMain} actually implies that
	\begin{equation*}
		| \inner{\vec{x}_i, \vec{x}_j}|^2 = |\inner{ k_{z_i}, k_{z_j} }|^2 = 1 - \rho^2(z_i,z_j)
	\end{equation*}
	where
	\begin{equation*}
		\rho(z,w):= \left| \frac{z-w}{ 1 - \overline{w}z} \right|
	\end{equation*}
	denotes the \emph{pseudohyperbolic metric} on $\D$.  In other words, we can
	obtain the pseudohyperbolic distances $\rho(z_i,z_j)$ \emph{directly} from the data
	$\vec{x}_1, \vec{x}_2, \ldots, \vec{x}_n$:
	\begin{equation}\label{eq-DistanceFormula}
		\rho(z_i,z_j) = \sqrt{1 - | \inner{\vec{x}_i,\vec{x}_j}|^2} .
	\end{equation}

	Let us briefly summarize the contents of this note.
	The proof of Theorem \ref{TheoremMain} is contained in Section \ref{SectionProof}.
	In Section \ref{SectionGeometry}, we use elementary hyperbolic geometry and some of the preceding geometric
	observations to construct an illuminating example.  Sections \ref{Section2x2} and \ref{Section3x3} 
	provide complete analyses of the $2\times 2$ and $3 \times 3$ cases, respectively.
	Among other things, we prove that every complex symmetric operator on a three-dimensional Hilbert space is unitarily equivalent
	to a direct sum of truncated Toeplitz operators (Theorem \ref{TheoremDS}).
	Section \ref{SectionNecessary} concerns a simple necessary condition for a matrix to be
	unitarily equivalent to an analytic truncated Toeplitz operator.
	We conclude this note in Section \ref{SectionOpen} with a number of open problems that we hope
	will spur further research into this topic.
	
	\medskip
	\noindent\textbf{Acknowledgments}:  We relied heavily upon numerical experiments
	to test several conjectures (which eventually led to proofs of Theorems \ref{Theorem3x3} and \ref{TheoremDS}).  
	We wish to thank J.~Tener for independently confirming several of our numerical observations.

\section{Proof of Theorem \ref{TheoremMain}}\label{SectionProof}	
	
	To prove Theorem \ref{TheoremMain}, we require a few preliminaries.
	We first note that for each disk automorphism $\psi$, the 
	weighted composition operator $U_{\psi}: \K_{\Theta} \to \K_{\Theta \circ \psi}$
	defined by $U_{\psi} f := \sqrt{\psi'} (f \circ \psi)$
	is unitary and furnishes a bijection between the set of analytic
	truncated Toeplitz operators on $\K_{\Theta}$ and those on $\K_{\Theta\circ\psi}$.
	In particular,
	\begin{equation}\label{eq-Spatial}
		A_{\phi}^{\Theta} \cong A_{\phi \circ \psi}^{\Theta\circ\psi}
	\end{equation}
	holds for all $\phi$ in $H^{\infty}$ \cite[Prop.~4.1]{TTOSIUES}.
	Our next ingredient is the following simple lemma, which is
	inspired by the proof of \cite[Thm.~2]{UECSMGC}.  

	\begin{Lemma}\label{LemmaMain}
		Let $\mathcal{X}$ and $\mathcal{Y}$ be $n$-dimensional Hilbert spaces.
		If $x_1,x_2,\ldots,x_n \in \mathcal{X}$ and $y_1,y_2,\ldots,y_n \in \mathcal{Y}$ are linearly independent sets of
		unit vectors such that $\inner{y_i,y_j} \neq 0$
		for $1 \leq i,j \leq n$, then the following are equivalent:
		\begin{enumerate}\addtolength{\itemsep}{0.5\baselineskip}
			\item There exist unimodular constants $\alpha_1, \alpha_2, \ldots, \alpha_n$
				and a unitary operator $U:\mathcal{X} \to \mathcal{Y}$ such that $Ux_i = \alpha_i y_i$ for $1 \leq i \leq n$.
	
			\item There exist unimodular constants $\alpha_1, \alpha_2, \ldots, \alpha_n$ such that
				\begin{equation*}
					\inner{x_i,x_j} = \alpha_i \overline{ \alpha_j }    \inner{ y_i,y_j}
				\end{equation*}
				for $1 \leq i \leq j \leq n$.
	
			\item For some fixed $k$,
				\begin{equation}\label{eq-Triple}
					\inner{x_k,x_i} \inner{x_i,x_j} \inner{x_j,x_k}     
					= \inner{y_k,y_i}\inner{y_i,y_j} \inner{ y_j,y_k} 
				\end{equation}
				holds for $1 \leq i \leq j \leq n$.
				
			\item The condition \eqref{eq-Triple} holds for $1 \leq i \leq j \leq k \leq n$.
		\end{enumerate}
	\end{Lemma}

	\begin{proof}
		The implications (i) $\Leftrightarrow$ (ii) $\Rightarrow$ (iv) $\Rightarrow$ (iii) are obvious.
		We therefore prove only (iii) $\Rightarrow$ (ii).  Setting $j=k$ in \eqref{eq-Triple} reveals that
		$|\inner{x_i,x_k}| = |\inner{y_i,y_k}| \not = 0$ for $1\leq i\leq n$.  In light of this, we conclude from
		\eqref{eq-Triple} that $|\inner{x_i,x_j}| = |\inner{y_i,y_j}| \not = 0$ for $1 \leq i,j \leq n$ whence the constants
		\begin{equation*}
			\beta_{ij} = \frac{ \inner{ x_i,x_j} }{ \inner{y_i,y_j} }
		\end{equation*}
		are unimodular and satisfy $\beta_{ij} = \beta_{ik} \beta_{kj}$
		for $1 \leq i,j\leq n$.  Now define $\alpha_i = \beta_{ik}$ for $1 \leq i \leq n$ 
		and observe that
		\begin{equation*}
			\inner{x_i,x_j} = \beta_{ij}\inner{y_i,y_j} = \beta_{ik} \beta_{kj} \inner{y_i,y_j}
			=  \alpha_i\overline{\alpha_j}  \inner{y_i,y_j}.\qedhere
		\end{equation*}
	\end{proof}
	
	With additional effort,
	one can remove the hypothesis that $\inner{y_i,y_j} \neq 0$ for $1 \leq i,j \leq n$.
	However, this is unnecessary in our case since eigenvectors of an analytic truncated
	Toeplitz operator cannot be orthogonal to each other.

\begin{proof}[Proof of Theorem \ref{TheoremMain}]
	$(\Rightarrow)$
	Suppose that $UM = A_{\phi}^{\Theta} U$ where $\phi \in H^{\infty}$
	and $U:\C^n \to \K_{\Theta}$ is unitary.
	Next observe that if $M\vec{x}_i = \lambda_i \vec{x}_i$, then
	$A_{\phi}^{\Theta} (U\vec{x}_i) = \lambda_i (U\vec{x}_i)$ for $1 \leq i \leq  n$.
	Thus there exists an enumeration
	$z_1,z_2,\ldots,z_n$ of the zeros of $\Theta$ and unimodular constants
	$\alpha_1,\alpha_2,\ldots,\alpha_n$ so that $U\vec{x}_i = \alpha_i \tilde{k}_{z_i}$ for $1 \leq i \leq n$.
	Without loss of generality we may assume that $z_n =0$.  Indeed, otherwise let
	$\psi$ be an automorphism of $\D$ satisfying $\psi(0) = z_n$ and note that
	$M \cong A_{\phi}^{\Theta} \cong A_{\phi \circ \psi}^{\Theta\circ\psi}$ by \eqref{eq-Spatial}
	(by precomposing with a rotation, we may also assume that $0 < z_1 < 1$ if we wish).
	By Lemma \ref{LemmaMain}, it follows that
	\begin{align*}
		\inner{\vec{x}_n,\vec{x}_i} \inner{\vec{x}_i,\vec{x}_j} \inner{\vec{x}_j,\vec{x}_n}
		&= \inner{ \tilde{k}_{0}, \tilde{k}_{z_i} } \inner{\tilde{k}_{z_i}, \tilde{k}_{z_j}} \inner{ \tilde{k}_{z_j}, \tilde{k}_0} \\
		&= \inner{ k_{z_i}, k_0 } \inner{ k_{z_j}, k_{z_i}} \inner{ k_0, k_{z_i}} \\
		&= \frac{ (1 - |z_i|^2)(1-|z_j|^2) }{ 1 - \overline{z_j}z_i}
	\end{align*}
	for $1 \leq i,j \leq n$, which is the desired condition \eqref{eq-TripleExplicit}.\medskip

	\noindent$(\Leftarrow)$
	Suppose that there exist distinct points
	$z_1,z_2,\ldots,z_{n-1}$ in $\D$ such that \eqref{eq-TripleExplicit} holds and let $\Theta$ be a Blasckhe
	product of order $n$ having its zeros at the points $z_1,z_2,\ldots,z_{n-1}$ and $z_n = 0$.
	It follows from \eqref{eq-TripleExplicit} that
	\begin{equation*}
		\inner{\vec{x}_n,\vec{x}_i} \inner{\vec{x}_i,\vec{x}_j} \inner{\vec{x}_j,\vec{x}_n} =
		\inner{ \tilde{k}_{0}, \tilde{k}_{z_i} } \inner{\tilde{k}_{z_i}, \tilde{k}_{z_j}} \inner{ \tilde{k}_{z_j}, \tilde{k}_0}
	\end{equation*}
	for $1 \leq i,j \leq n$
	whence by Lemma \ref{LemmaMain} there exists a unitary operator $U:\C^n \to \K_{\Theta}$ and unimodular
	constants $\alpha_1,\alpha_2,\ldots,\alpha_n$ such that
	$Ux_i = \alpha_i \tilde{k}_{z_i}$ for $1 \leq i \leq n$.
	If $\phi$ is a polynomial satisfying $\phi(z_i) = \lambda_i$ for $1 \leq i \leq n$, then
	$A_{\phi}^{\Theta} \tilde{k}_{z_i} =  \lambda_i \tilde{k}_{z_i}$ whence
	\begin{equation*}
		UM\vec{x}_i
		= \lambda_i U\vec{x}_i
		= \phi(z_i) \alpha_i \tilde{k}_{z_i}
		= \alpha_i A_{\phi}^{\Theta} \tilde{k}_{z_i}
		= A_{\phi}^{\Theta} U \vec{x}_i
	\end{equation*}
	for $1 \leq i \leq n$ so that  $M \cong A_{\phi}^{\Theta}$.
\end{proof}

\section{Hyperbolic geometry}\label{SectionGeometry}

	It is easy to construct matrices which are not unitarily equivalent
	to any analytic truncated Toeplitz operator.  For instance, any matrix having a pair of 
	orthogonal eigenvectors suffices.  On the other hand, what can be
	said about matrices having distinct eigenvalues and such that no pair of eigenvectors is orthogonal?  
	Using some basic hyperbolic geometry, we can construct a
	family of such matrices which are not unitarily equivalent to an analytic truncated Toeplitz operator.

	\begin{Example}\label{Example4x4} 
		Let us begin by searching for a matrix $M \in {\bf M}_n(\C)$ having distinct eigenvalues
		and whose corresponding normalized eigenvectors $\vec{x}_1, \vec{x}_2, \ldots, \vec{x}_n$
		are such that the condition \eqref{eq-TripleExplicit} does not hold for any points
		$z_1,z_2,\ldots,z_n$ in $\D$.
		Fix $0 < g < 1$ and let $G$ be the $n \times n$ matrix with entries
		\begin{equation*}
			G_{ij} = 
			\begin{cases}
				1 & \text{if $i = j$}, \\
				g & \text{if $i \neq j$}.
			\end{cases}
		\end{equation*}
		Since $G = g \vec{u} \vec{u}^* + (1-g)I$ where $\vec{u} = (1,1,\ldots,1)$ and
		$I$ denotes the $n \times n$ identity matrix, it follows that the eigenvalues of $G$ are precisely
		\begin{equation*}
			\underbrace{ 1-g,\quad 1-g,\ldots,\quad 1-g}_{\text{$n-1$ times}},\quad 1+ (n-1)g,
		\end{equation*}
		whence $G$ is positive definite.  A routine computation confirms that the entries of the 
		positive square root $X$ of $G$ are given by
		\begin{equation*}
			X_{ij} = 
			\begin{cases}
				\dfrac{(n-1)\sqrt{1-g} + \sqrt{1+(n-1)g}}{n} & \text{if $i = j$}, \\[8pt]
				\dfrac{ - \sqrt{1-g} + \sqrt{1+ (n-1)g} }{n} & \text{if $i \neq j$},
			\end{cases}
		\end{equation*}
		and hence each column $\vec{x}_i$ of $X = (\vec{x}_1 | \vec{x}_2 | \cdots | \vec{x}_n)$ is a unit vector.
		If $\lambda_1,\lambda_2,\ldots,\lambda_n$ are distinct complex numbers and 
		$D = \diag(\lambda_1,\lambda_2,\ldots,\lambda_n)$, then the matrix
		$M = XDX^{-1}$ satisfies $X^*X = G$ and $M \vec{x}_i = \lambda_i \vec{x}_i$ for $1 \leq i \leq n$.
		
		Suppose toward a contradiction that $M$ is unitarily equivalent to
		$A_{\phi}^{\Theta}$ for some $\phi \in H^{\infty}$ and some Blaschke
		product $\Theta$ having distinct zeros $z_1,z_2,\ldots,z_n$ (the hypothesis that $D$ has distinct
		eigenvalues ensures that the $z_i$ are distinct).  By \eqref{eq-DistanceFormula}
		it follows that $\rho(z_i,z_j) = \sqrt{1-g^2}$ for $i \neq j$.  
		Since the \emph{hyperbolic metric} (also called the \emph{Poincar\'e metric}) on $\D$ satisfies
		\begin{equation*}
			\psi(z,w) := \log \frac{1 + \rho(z,w)}{1 - \rho(z,w)},
		\end{equation*}
		it follows that
		\begin{equation}\label{eq-Distances}
			\psi(z_i,z_j) = 
			\begin{cases}
				r & \text{if $i \neq j$}, \\
				0 & \text{if $i = j$},
			\end{cases}
		\end{equation}
		where $r =2\tanh^{-1} \sqrt{1-g^2}$ \cite[p.~4-5]{MR2261424}.  This is impossible if $n \geq 5$.
		Indeed, suppose that $z_1,z_2,z_3,z_4,z_5$ satisfy \eqref{eq-Distances}.
		Recalling that a pseudohyperbolic circle is also a Euclidean circle \cite[p.~3]{MR2261424}, we obtain circles
		$\Gamma_1,\Gamma_2$, both of the same pseudohyperbolic radius, 
		such that $z_1,z_3,z_4,z_5 \in \Gamma_2$ and $z_2,z_3,z_4,z_5 \in \Gamma_1$.  
		This implies that $\{z_3,z_4,z_5\} \subseteq \Gamma_1 \cap \Gamma_2$, which contradicts the fact
		that two distinct Euclidean circles can meet in at most two points.		
		
		A slightly more complicated proof shows that \eqref{eq-Distances} is also impossible if $n =4$.
		Although the corresponding result is obvious in the Euclidean plane, to prove it in the hyperbolic plane one 
		first recalls that the \emph{Poincar\'e model} of the hyperbolic plane satisfies
		Hilbert's axioms \cite[Sect.~39]{Hartshorne}.
		One can then prove that Propositions I.2-I.22 and I.24-I.28 of Euclid's \emph{Elements} \cite{Euclid} 
		can be obtained in the Poincar\'e model \cite[Thm.~10.4]{Hartshorne} and then proceed as in the Euclidean case.
	\end{Example}

	It actually turns out that any matrix $3 \times 3$ or larger produced using the method of Example \ref{Example4x4} 
	cannot be a complex symmetric operator.
	In particular, no such matrix can be unitarily equivalent to a truncated Toeplitz operator, analytic or otherwise.
	First note that the entries of $W = X^{-1}$ are given by
	\begin{equation*}
		W_{ij}=
		\begin{cases}
			\frac{1}{n}\left( \frac{n-1}{ \sqrt{1-g} } + \frac{1}{\sqrt{1 + (n-1)g}} \right) & \text{if $i=j$},\\
			\frac{1}{n}\left( -\frac{1}{ \sqrt{1-g} } + \frac{1}{\sqrt{1 + (n-1)g}} \right) & \text{if $i\neq j$}.
		\end{cases}
	\end{equation*}
	Writing $W = ( \vec{w}_1 | \vec{w}_2| \cdots | \vec{w}_n)$ in column-by-column format, we find that
	\begin{equation*}
		\norm{\vec{w}_i} = \sqrt{ \frac{1 + (n-2)g}{1 + [ (n-2) - (n-1)g]g} }
	\end{equation*}
	for $i = 1,2,\ldots,n$.  Upon dividing $W$ by the preceding quantity we obtain the matrix 
	$Y = (\vec{y}_1 | \vec{y}_2 | \cdots | \vec{y}_n)$ whose entries are given by
	\begin{equation*}
		Y_{ij}=
		\begin{cases}
			\frac{ \sqrt{1-g} + (n-1) \sqrt{1 + (n-1)g} }{ n \sqrt{1 + (n-2)g} } & \text{if $i=j$},\\[9pt]
			\frac{ \sqrt{1-g} - \sqrt{1 + (n-1)g} }{ n \sqrt{1 + (n-2)g} } & \text{if $i\neq j$},
		\end{cases}
	\end{equation*}
	and whose columns $\vec{y}_1, \vec{y}_2 , \ldots, \vec{y}_n$ are unit vectors.
	Since $M = XDX^{-1}$ and $X$ is self-adjoint, 
	it follows that $M^* = Y \overline{D} Y^{-1}$
	Therefore $\vec{y}_1, \vec{y}_2, \ldots, \vec{y}_n$ are unit eigenvectors of $M^*$
	corresponding to the eigenvalues $\overline{\lambda_1}, \overline{\lambda_2}, \ldots, \overline{\lambda_n}$.
	If $n \geq 3$, then 
	\begin{equation*}
		|\inner{ \vec{x}_i, \vec{x}_j} | = g \neq \frac{g}{1 + (n-2)g} = | \inner{ \vec{y}_i, \vec{y}_j }|,
	\end{equation*}
	whence $M$ is not a complex symmetric operator by \cite[Thm.~1]{UECSMGC}.

\section{The $2 \times 2$ case}\label{Section2x2}

	If $M \in {\bf M}_2(\C)$ and $\Theta$ is a Blaschke product of order two, then 
	there exists a truncated Toeplitz operator on $\K_{\Theta}$ which is unitarily equivalent to $M$ \cite[Thm.~5.2]{TTOSIUES}.
	However, if one insists upon using \emph{analytic} symbols, then things are quite different.

	\begin{Corollary}\label{Corollary2x2}
		If $M \in {\bf M}_2(\C)$, then $M$ is unitarily equivalent to an analytic truncated Toeplitz
		operator if and only if either 
		\begin{enumerate}\addtolength{\itemsep}{0.25\baselineskip}
			\item $M$ is a multiple of the identity,
			\item $M$ is not normal.
		\end{enumerate}
	\end{Corollary}

	\begin{proof}
		$(\Rightarrow)$ Suppose that $M \cong A_{\phi}^{\Theta}$ for some Blaschke product
		$\Theta$ of order two and some $\phi$ in $H^{\infty}$. \medskip
		
		\noindent\textbf{Case 1}:  If $\Theta$ has distinct zeros $z_1,z_2$, then
		$\tilde{k}_{z_1}$ and $\tilde{k}_{z_2}$ are linearly independent eigenvectors of $A_{\phi}$ corresponding to the eigenvalues
		$\phi(z_1)$ and $\phi(z_2)$, respectively.  In particular, $M$ is diagonalizable.
		If $\phi(z_1) = \phi(z_2)$, then $M$ is a multiple of the identity.
		If $\phi(z_1) \neq \phi(z_2)$, then $A_{\phi}$ is not normal since
		$\inner{\tilde{k}_{z_1}, \tilde{k}_{z_2}}  \neq 0$.
		\medskip
		
		\noindent\textbf{Case 2}:
		If $\Theta$ has a repeated root, then by \eqref{eq-Spatial} we may assume that $\Theta(z) = z^2$.
		Thus $M$ is unitarily equivalent to a lower-triangular Toeplitz matrix.    Since a normal triangular matrix is
		diagonal \cite[p.~133]{Axler}, it follows that $M$ is either non-normal or a 
		multiple of the identity matrix.
	
		\medskip
		\noindent$(\Leftarrow)$
		If $M$ is a  multiple of the identity matrix, then there is nothing to prove.  Suppose now that
		$M$ is not normal.  If $M$ has repeated eigenvalues, then Schur's theorem on unitary triangularization \cite[Thm.~2.3.1]{HJ}
		asserts that $M$ is unitarily equivalent to a lower triangular Toeplitz matrix and hence an analytic
		truncated Toeplitz operator.  Suppose that $M$ has distinct eigenvalues with corresponding
		unit eigenvectors $\vec{x}_1$ and $\vec{x}_2$.  The only nontrivial condition in \eqref{eq-TripleExplicit}
		that needs to be satisfied is $1 - |z_1|^2 = |\inner{\vec{x}_2,\vec{x}_1}|^2 \neq 0$, which has many solutions.
	\end{proof}
	
	\begin{Example}
		Consider the non-normal matrix
		\begin{equation*}
			M = \minimatrix{1}{2}{0}{3},
		\end{equation*}
		whose eigenvalues are $\lambda_1 = 1$ and $\lambda_2 = 3$.  Corresponding
		unit eigenvectors of $M$ corresponding to these eigenvalues are 
		\begin{equation*}
			\vec{x}_1 = (1,0), \qquad \vec{x}_2 = (\tfrac{1}{\sqrt{2}}, \tfrac{1}{\sqrt{2}}).
		\end{equation*}
		Guided by the proof of Theorem \ref{TheoremMain}, we set 
		\begin{equation*}
			z_1 = \sqrt{1-|\inner{\vec{x}_2,\vec{x}_1}|^2} = \tfrac{1}{\sqrt{2}}
		\end{equation*}
		and $z_2 = 0$.  Next, we search for a polynomial $\phi(z)$ such that 
		\begin{equation*}
			\phi(\tfrac{1}{\sqrt{2}}) = 1, \qquad \phi(0) = 3.
		\end{equation*}
		One such polynomial is $\phi(z) = 3 - 2\sqrt{2}z$, from which we conclude that $M$ is unitarily equivalent to
		$A_{\phi}^{\Theta}$ where 
		\begin{equation*}
			\Theta(z) = z\left( \frac{z - \tfrac{1}{\sqrt{2}}}{1 - \tfrac{1}{\sqrt{2}} z} \right).
		\end{equation*}
	\end{Example}	

	Following J.~Tener \cite{Tener}, we say that a matrix is UECSM if it is unitarily equivalent
	to a complex symmetric matrix (i.e., it represents a complex symmetric operator with 
	respect to some orthonormal basis).
	Although there are many proofs of the following result
	(see \cite[Cor.~3]{UECSMGC},  \cite[Cor.~3.3]{CFT}, \cite[Ex.~6]{CSOA}, \cite{UECSMMC}, \cite{UET},
	\cite[Cor.~1]{SNCSO}, \cite[p.~477]{McIntosh}, or \cite[Cor.~3]{Tener}),
	we feel compelled to provide yet another.

	\begin{Corollary}\label{Corollary2x2UECSM}
		If $M \in {\bf M}_2(\C)$, then $M$ is UECSM.
	\end{Corollary}

	\begin{proof}
		Let $M$ be a $2 \times 2$ matrix.  If $M$ is either a multiple of the identity
		or normal, then $M$ is trivially UECSM (by the Spectral Theorem).
		Otherwise, $M$ is unitarily equivalent to an analytic truncated Toeplitz operator 
		by \eqref{Corollary2x2} and hence UECSM \cite{CCO,CSOA}.
	\end{proof}

\section{The $3 \times 3$ case}\label{Section3x3}

	Although the $3 \times 3$ case is significantly more complicated than the $2 \times 2$ case,
	we are still able to arrive at a complete solution, including a
	simple computational criterion \eqref{eq-Determinant}.
	Moreover, we also show that every $3 \times 3$ complex symmetric matrix is unitarily equivalent to a direct sum of
	truncated Toeplitz operators (Theorem \ref{TheoremDS}).

	\begin{Theorem}\label{Theorem3x3}
		If $M \in {\bf M}_3(\C)$ has distinct eigenvalues $\lambda_1,\lambda_2,\lambda_3$ with corresponding normalized eigenvectors 
		$\vec{x}_1, \vec{x}_2, \vec{x}_3$ satisfying $\inner{\vec{x}_i,\vec{x}_j} \neq 0$
		for $1 \leq i,j \leq 3$, then the following are equivalent:
		\begin{enumerate}\addtolength{\itemsep}{0.5\baselineskip}
			\item $M$ is unitarily equivalent to an analytic truncated Toeplitz operator,
			\item $M$ is unitarily equivalent to a complex symmetric matrix,
			\item The condition
				\begin{equation}\label{eq-Determinant}
					\det X^*X = (1 - | \inner{\vec{x}_1 , \vec{x}_2 }|^2)
					(1 - | \inner{\vec{x}_2 , \vec{x}_3 }|^2)(1 - | \inner{\vec{x}_3 , \vec{x}_1 }|^2)
				\end{equation}
				holds, where $X = ( \vec{x}_1 | \vec{x}_2 | \vec{x}_3)$ is the matrix having
				$\vec{x}_1, \vec{x}_2, \vec{x}_3$ as its columns.
		\end{enumerate}
	\end{Theorem}

	\begin{proof}
		(i) $\Rightarrow$ (ii) This implication is well-known \cite{CCO, CSOA, Sarason} (i.e., every
		truncated Toeplitz operator is a complex symmetric operator).
		\medskip
		
		\noindent (ii) $\Rightarrow$ (iii).  Without loss of generality, we may assume that $M = M^t$
		is a complex symmetric matrix.  In this case
		$\overline{ \vec{x}_1}, \overline{ \vec{x}_2}, \overline{ \vec{x}_3}$ are unit eigenvectors of $M^*$ corresponding
		to the eigenvalues $\overline{\lambda_1}, \overline{\lambda_2}, \overline{\lambda_3}$.  In light of the fact that 
		\begin{equation*}
			\lambda_i \inner{ \vec{x}_i, \overline{ \vec{x}_j} } = \inner{ M\vec{x}_i, \overline{\vec{x}_j} } 
			= \inner{ \vec{x}_i, M^*\overline{\vec{x}_j} }
			= \inner{ \vec{x}_i , \overline{ M \vec{x}_j} } = \inner{ \vec{x}_i, \overline{ \lambda_j \vec{x}_j} } 
			=\lambda_j \inner{ \vec{x}_i, \overline{ \vec{x}_j} },
		\end{equation*}
		we see that $\inner{ \vec{x}_i , \overline{ \vec{x}_j} } = 0$ for $i \neq j$.
		
		We claim that
		\begin{align}
			| \inner{ \vec{x}_1, \overline{ \vec{x}_1} } |^2 &= (1 - |\inner{\vec{x}_1,\vec{x}_2}|^2)(1 - |\inner{\vec{x}_3,\vec{x}_1}|^2), 
				\label{eq-Choice1}\\
			| \inner{ \vec{x}_2, \overline{ \vec{x}_2} } |^2 &= (1 - |\inner{\vec{x}_1,\vec{x}_2}|^2)
				(1 - |\inner{\vec{x}_2,\vec{x}_3}|^2), \label{eq-Choice2}\\
			| \inner{ \vec{x}_3, \overline{ \vec{x}_3} } |^2 &= (1 - |\inner{\vec{x}_2,\vec{x}_3}|^2)(1 - |\inner{\vec{x}_3,\vec{x}_1}|^2).
				\label{eq-Choice3}
		\end{align}
		We prove only \eqref{eq-Choice2}, since \eqref{eq-Choice1} and \eqref{eq-Choice3} can be proven using a similar method.
		First note that $\overline{\vec{x}_1}$ and $\overline{\vec{x}_2}$ are linearly independent since $\lambda_1 \neq \lambda_2$.
		Since $\inner{ \vec{x}_3, \overline{ \vec{x}_1} } = \inner{ \vec{x}_3, \overline{ \vec{x}_2} } = 0$, it follows that 
		$\{ \overline{\vec{x}}_1, \overline{\vec{x}}_2, \vec{x}_3\}$ is a basis for $\C^3$.  Thus the basis
		$\{ \vec{e}_1, \vec{e}_2, \vec{e}_3\}$ defined by setting $\vec{e}_1 = \overline{ \vec{x}_1}$,
		$\vec{e}_3 = \vec{x}_3$, and
		\begin{equation*}
			\vec{e}_2 
			= \frac{ \overline{\vec{x}_2} - \inner{ \overline{ \vec{x}_2}, \vec{e}_1 } \vec{e}_1}
				{\norm{ \overline{\vec{x}_2} - \inner{ \overline{ \vec{x}_2}, \vec{e}_1 } \vec{e}_1}} 
			= \frac{ \overline{ \vec{x}_2} - \inner{ \vec{x}_1 , \vec{x}_2 } \overline{ \vec{x}_1} }{ \sqrt{1 - | \inner{ \vec{x}_1, \vec{x}_2}|^2}},
		\end{equation*}
		is orthonormal.  Since $\vec{x}_2$ is a unit vector it follows that
		\begin{align*}
			1
			&= | \inner{ \vec{x}_2, \vec{e}_1} |^2 + | \inner{ \vec{x}_2, \vec{e}_2} |^2 + | \inner{ \vec{x}_2, \vec{e}_3} |^2 \\
			&= | \inner{ \vec{x}_2, \overline{\vec{x}_1}} |^2 + \left| \left< \vec{x}_2, 
				\frac{ \overline{ \vec{x}_2} - \inner{ \vec{x}_1 , \vec{x}_2 } \overline{ \vec{x}_1} }
				{ \sqrt{1 - | \inner{ \vec{x}_1, \vec{x}_2}|^2}} \right> \right|^2 + | \inner{ \vec{x}_2, \vec{x}_3} |^2 \\
			&= \frac{| \inner{ \vec{x}_2, \overline{ \vec{x}_2}}|^2 }{ 1 - | \inner{ \vec{x}_1, \vec{x}_2} |^2}+ | \inner{ \vec{x}_2, \vec{x}_3} |^2,
		\end{align*}
		which is equivalent to \eqref{eq-Choice2}.  
		
		Since $X^t X = \diag(\inner{ \vec{x}_1, \overline{\vec{x}_1} },\inner{ \vec{x}_2, \overline{\vec{x}_2} },
		\inner{ \vec{x}_3, \overline{\vec{x}_3} })$, it follows from \eqref{eq-Choice1}, \eqref{eq-Choice2},
		and \eqref{eq-Choice3} that
		\begin{align*}
			\det X^*X
			&= |\det X|^2 \\
			&= |\det X^t X| \\
			&= |\inner{ \vec{x}_1, \overline{\vec{x}_1} }\inner{ \vec{x}_2, \overline{\vec{x}_2} } \inner{ \vec{x}_3, \overline{\vec{x}_3} }| \\
			&= (1 - | \inner{\vec{x}_1 , \vec{x}_2 }|^2)(1 - | \inner{\vec{x}_2 , \vec{x}_3 }|^2)(1 - | \inner{\vec{x}_3 , \vec{x}_1 }|^2).
		\end{align*}
		This yields the desired condition \eqref{eq-Determinant}.
		\medskip

		\noindent (iii) $\Leftrightarrow$ (i). 
		In light of Theorem \ref{TheoremMain} and its proof, it follows that $M$ is 
		unitarily equivalent to an analytic truncated Toeplitz operator if and only if there exist
		$z_1 \in (0,1)$ and $z_2 \in \D$ such that \eqref{eq-TripleExplicit} holds for 
		$1 \leq i,j,k \leq 3$.  Moreover, by \eqref{eq-Moduli} we know that $z_1$ and $z_2$
		must be given by
		\begin{equation}\label{eq-Zvalues}
			z_1 = \sqrt{ 1 - | \inner{ \vec{x}_1, \vec{x}_3} |^2},\qquad
			z_2 = \sqrt{1 - |\inner{ \vec{x}_2, \vec{x}_3}|^2} e^{it}
		\end{equation}
		for some $t \in \R$.  In other words, $M$ is
		unitarily equivalent to an analytic truncated Toeplitz operator
		if and only if there exists a real number $t$ such that the numbers $z_1,z_2$,
		as defined by \eqref{eq-Zvalues}, satisfy \eqref{eq-TripleExplicit}:
		\begin{equation*}
			\inner{ \vec{x}_3,\vec{x}_1} \inner{ \vec{x}_1, \vec{x}_2} \inner{ \vec{x}_2, \vec{x}_3}
			= \frac{ | \inner{ \vec{x}_1, \vec{x}_3}|^2 | \inner{ \vec{x}_2 , \vec{x}_3} | ^2}{1 - \overline{z_2}z_1},
		\end{equation*}
		which is equivalent to
		\begin{equation*}
			\overline{z_2} z_1
			= 1- \frac{  \inner{ \vec{x}_1, \vec{x}_3}  \inner{ \vec{x}_3 , \vec{x}_2}}{\inner{ \vec{x}_1,\vec{x}_2} }.
		\end{equation*}
		Substituting \eqref{eq-Zvalues} into the preceding we obtain
		\begin{equation*}
			e^{-it} \sqrt{ 1 - | \inner{ \vec{x}_1, \vec{x}_3} |^2}\sqrt{ 1 - | \inner{ \vec{x}_2, \vec{x}_3} |^2}
			= 1- \frac{  \inner{ \vec{x}_1, \vec{x}_3}  \inner{ \vec{x}_3 , \vec{x}_2}}{\inner{ \vec{x}_1,\vec{x}_2} },
		\end{equation*}
		which has a solution $t \in \R$ if and only if
		\begin{equation*}
			(1 - | \inner{ \vec{x}_1, \vec{x}_3} |^2 )( 1 - | \inner{ \vec{x}_2, \vec{x}_3} |^2)
			= \left|1- \frac{  \inner{ \vec{x}_1, \vec{x}_3}  \inner{ \vec{x}_3 , \vec{x}_2}}{\inner{ \vec{x}_1,\vec{x}_2} } \right|^2.
		\end{equation*}
		Expanding the preceding, we obtain
		\begin{align}
			&| \inner{ \vec{x}_1, \vec{x}_3} \inner{ \vec{x}_3, \vec{x}_2} |^2 
			+| \inner{ \vec{x}_2, \vec{x}_1} \inner{ \vec{x}_1, \vec{x}_3} |^2
			+| \inner{ \vec{x}_3, \vec{x}_2} \inner{ \vec{x}_2, \vec{x}_1} |^2 \nonumber \\
			&\qquad =
			|\inner{\vec{x}_1,\vec{x}_2}\inner{\vec{x}_2,\vec{x}_3}\inner{\vec{x}_3,\vec{x}_1}|^2 
			+ 2 \Re \inner{\vec{x}_1,\vec{x}_2}\inner{\vec{x}_2,\vec{x}_3}\inner{\vec{x}_3,\vec{x}_1}.\label{eq-Crazy}
		\end{align}
		Adding the quantity
		\begin{equation*}
			1 - |\inner{\vec{x}_1,\vec{x}_2}|^2 -  |\inner{\vec{x}_2,\vec{x}_3}|^2 -  |\inner{\vec{x}_3,\vec{x}_1}|^2
		\end{equation*}
		to both sides of \eqref{eq-Crazy} yields
		\begin{align*}
			1 - |\inner{\vec{x}_1,\vec{x}_2}|^2 -  |\inner{\vec{x}_2,\vec{x}_3}|^2 -  |\inner{\vec{x}_3,\vec{x}_1}|^2
			+ 2 \Re\inner{\vec{x}_1,\vec{x}_2}\inner{\vec{x}_2,\vec{x}_3}\inner{\vec{x}_3,\vec{x}_1}\\
			= (1 - |\inner{\vec{x}_1,\vec{x}_2}|^2)(1 - |\inner{\vec{x}_2,\vec{x}_3}|^2)(1 - |\inner{\vec{x}_3,\vec{x}_1}|^2),
		\end{align*}
		which is equivalent to \eqref{eq-Determinant}. 
	\end{proof}

	\begin{Theorem}\label{TheoremDS}
		If $M \in {\bf M}_3(\C)$ is complex symmetric, then $M$ is unitarily equivalent to a direct sum of truncated Toeplitz operators.
	\end{Theorem}

	In order to proceed with the proof of Theorem \ref{TheoremDS}, we require the 
	following lemma from \cite{UET}.  The proof in the preprint \cite{UET} is long and involved, since it requires
	invoking the fact that a matrix $7 \times 7$ or smaller which is unitarily equivalent to its transpose is UECSM
	(this can fail for matrices $8 \times 8$ and larger).  We provide a much simpler proof below.

	\begin{Lemma}\label{LemmaTraceTest}
		If $M \in{\bf M}_3(\C)$, then $M$ is UECSM if and only if
		\begin{equation}\label{eq-TraceTest}
			\tr M^*M^2{M^*}^2 M = \tr M {M^*}^2 M^2 M^*.
		\end{equation}
	\end{Lemma}
	
	\begin{proof}
		By Corollary \ref{Corollary2x2UECSM} we may assume that $M$ is irreducible since otherwise
		$M$ is obviously UECSM.  
		We use the term \emph{irreducible} in the operator-theoretic sense.
		Namely, a matrix $T \in M_n(\C)$ is called irreducible if $T$ is not unitarily
		equivalent to a direct sum $A \oplus B$ where $A \in M_d(\C)$ and $B \in M_{n-d}(\C)$
		for some $1 < d < n$.  Equivalently, $T$ is \emph{irreducible} if and only if the only normal matrices
		commuting with $T$ are multiples of the identity. 	

		We claim that $M$ is UECSM if and only if $M \cong M^t$.  One direction is simple, 
		for if $M$ is UECSM, then there exist a unitary matrix 
		$U$ such that $U^*MU = (U^*MU)^t$.  In other words,
		$M(UU^t) = (UU^t)M^t$ whence $M \cong M^t$.  
		
		On the other hand if $M \cong M^t$, then we may write $M = UM^tU^*$ where $U$ is unitary.
		It follows that $M^t = \overline{U} M U^t$ whence $M U\overline{U} = U \overline{U} M$.
		Since $M$ is irreducible, $U\overline{U} = \alpha I$ for some unimodular constant $\alpha$.
		The preceding implies that $U = \alpha U^t$ whence $\alpha^2 = 1$.  However, $\alpha = -1$
		is impossible since every skew-symmetric matrix of odd dimension is singular. 
		Therefore $U = U^t$ is symmetric and unitary.  By Takagi's Factorization Theorem \cite{HJ}, we may write
		$U = V V^t$ where $V$ is unitary whence $V^*MV = (V^*MV)^t$.
		In other words, $M$ is UECSM.
	
		Having shown that $M$ is UECSM if and only if $M \cong M^t$, we need only show that \eqref{eq-TraceTest} holds
		if and only if $M \cong M^t$.  To this end, we recall that a 
		refinement by Sibirski{\u\i} \cite{Sibirskii} of a well-known result of Pearcy \cite{Pearcy}
		asserts that $A,B \in {\bf M}_3(\C)$ are unitarily equivalent
		if and only if $\Phi(A) = \Phi(B)$ 
		where $\Phi:M_3(\C)\to \C^7$ is the function defined by
		\begin{equation}\label{eq-Words}
			\Phi(X) = (\tr X, \, \tr  X^2,\, \tr X^3,\, \tr X^* X,\, \tr X^*X^2, \, \tr {X^*}^2 X^2, \, \tr X^* X^2 {X^*}^2X).
		\end{equation}
		It is easy to see that the first six traces in \eqref{eq-Words} are automatically
		equal for $X = M$ and $X = M^t$ \cite{UET}.  In other words, $M \cong M^t$ if and only if
		$\tr X^* X^2 {X^*}^2X$ yields the same value for $X = M$ and $X = M^t$.
		Using standard properties of the trace and transpose one sees that this condition is equivalent to \eqref{eq-TraceTest}.
	\end{proof}
	
	With the preceding lemma in hand, we are now ready to prove Theorem \ref{TheoremDS}.

	\begin{proof}[Pf.~of Theorem \ref{TheoremDS}]
		Suppose that $M$ is a $3 \times 3$ matrix.  After possibly scaling and adding a multiple of the identity, up to unitary equivalence
		$M$ falls into precisely one of the following classes:
		\begin{equation}\label{eq-Listed}
			\underbrace{\megamatrix{0}{0}{0}{a}{0}{0}{b}{c}{0},}_{\text{one distinct eigenvalue}} \qquad
			\underbrace{\megamatrix{0}{0}{0}{a}{0}{0}{b}{c}{1},}_{\text{two distinct eigenvalues}}\qquad
			\underbrace{\megamatrix{0}{0}{0}{a}{1}{0}{b}{c}{\lambda},}_{\text{three distinct eigenvalues}}
		\end{equation}
		where $a,b,c \in \C$ and $\lambda \neq 0,1$.
		\medskip
	
		\noindent\textbf{Case 1}:  If $M$ has one distinct eigenvalue, then without loss of generality we may assume that
		$M$ is of the form of the first matrix listed in \eqref{eq-Listed}.  Using Lemma \ref{LemmaTraceTest}
		and \texttt{Mathematica} it follows that $M$ is UECSM if and only if
		\begin{equation*}
			|a|^2|c|^2(|a|^2-|c|^2) = 0
		\end{equation*}
		(see \cite[Ex.~1]{SNCSO} and \cite[Ex.~1]{Tener} for other approaches).
		In other words, $M$ is UECSM if and only if either (i) $a=0$, (ii) $c=0$, or (iii) $|a| = |c|$.
		If either $a = 0$ or $c= 0$, then $M$ has rank one whence $M$ is unitarily equivalent to a truncated Toeplitz operator
		by \cite[Thm.~5.1]{TTOSIUES}.  On the other hand, if $|a| = |c|$, then conjugating $M$ by a diagonal unitary matrix shows that $M$ is
		unitarily equivalent to a Toeplitz matrix and hence $M$ represents a truncated Toeplitz operator on $\K_{z^3}$ with 
		respect to the orthonormal basis $\{1,z,z^2\}$.
		\medskip
		
		\noindent\textbf{Case 2}:  If $M$ has exactly two distinct eigenvalues, then we may assume that
		$M$ is of the form of the second matrix listed in \eqref{eq-Listed}.  Using Lemma \ref{LemmaTraceTest} and \texttt{Mathematica}, 
		it follows that $M$ is UECSM if and only if either (i) $a=0$, (ii) $b=c=0$, or (iii) $a,c\neq 0$ and
		\begin{equation}\label{eq-Magic}
			|b+ac|^2 = |c|^2(1+|c|^2).
		\end{equation}
		If $a =0$, then $M$ has rank one whence $M$ is unitarily equivalent to a truncated Toeplitz operator by \cite[Thm.~5.1]{TTOSIUES}.
		If $b=c=0$, then $M$ is the direct sum of a $2 \times 2$ and a $1 \times 1$ matrix.  By \cite[Thm.~5.2]{TTOSIUES},
		it follows that $M$ is unitarily equivalent to the direct sum of truncated Toeplitz operators.
		
		The third case is more difficult to handle.  Suppose that $a,c\neq 0$ and that \eqref{eq-Magic} holds.
		Upon conjugating $M$ by a suitable diagonal unitary matrix, we may further assume that $c > 0$ and
		$b+ac \geq 0$.  Let $\Theta$ denote the Blaschke product
		\begin{equation*}
			\Theta(z) = z^2\left( \frac{z-r}{1 - rz}\right),
		\end{equation*}
		which has a double root at $0$ and a simple zero at $r \in (0,1)$ which is to be determined.
		An orthonormal basis $\{e_1,e_2,e_3\}$ for $\K_{\Theta}$ is given by 
		\begin{equation}\label{eq-TMW}
			e_1(z) = 1, \qquad e_2(z) = z, \qquad
			e_3(z) = z^2 \frac{ \sqrt{1-r^2} }{1 - rz}.
		\end{equation}
		If $\phi(z) = \alpha z + \beta z^2$, then the matrix for $A_{\phi}^{\Theta}$ with respect to the 
		basis \eqref{eq-TMW} is given by
		\begin{equation}\label{eq-TTOM}
			\megamatrix{0}{0}{0}{\alpha}{0}{0}{\beta\sqrt{1-r^2}}{(\alpha+\beta r)\sqrt{1-r^2}}{ r(\alpha + \beta r)}.
		\end{equation}
		At this point it is easily verified that for
		\begin{equation*}
			\alpha = a, \qquad \beta = \frac{b}{c} \sqrt{1+c^2}, \qquad
			r = \frac{1}{\sqrt{1+c^2}},
		\end{equation*}
		the matrix \eqref{eq-TTOM} is precisely $M$.  In other words, $M$ is unitarily equivalent to an analytic
		truncated Toeplitz operator on $\K_{\Theta}$.
		\medskip
	
		\noindent\textbf{Case III}:  Suppose that $M$ has three distinct eigenvalues.
		If $M$ has a pair of eigenvectors which are orthogonal, then we may assume that
		$M$ is of the form of the third matrix listed in \eqref{eq-Listed} where $c = 0$.
		In this case, Lemma \ref{LemmaTraceTest} and \texttt{Mathematica} tell us that $M$ is UECSM if and only if
		$|ab|^2|\lambda - 1|^2 = 0$.  Since $\lambda \neq 1$ it follows that either $a=0$ or $b=0$.
		Both cases lead to the conclusion that $M$ is unitarily equivalent to the direct sum of a 
		$2 \times 2$ and a $1 \times 1$ matrix.  In particular, $M$ is unitarily equivalent to a direct
		sum of truncated Toeplitz operators.
		
		If $M$ has no pair of nonzero eigenvectors which are orthogonal to each other, then we
		may appeal to Theorem \ref{Theorem3x3} to conclude that $M$ is unitarily equivalent
		to an analytic truncated Toeplitz operator.
	\end{proof}

\section{A necessary condition}\label{SectionNecessary}

	Recall that every truncated Toeplitz operator is a complex symmetric operator and hence UECSM \cite{CCO,CSOA,Sarason}.
	Thus we might as well start with a complex symmetric matrix $M$ in the first place.
	Unlike Theorem \ref{TheoremMain}, the following proposition is phrased completely in terms
	of the initial data $\vec{x}_1,\vec{x}_2,\ldots,\vec{x}_n$.  In particular, there is no mention
	whatsoever of the unknowns $z_1,z_2,\ldots,z_{n-1}$.

	\begin{Corollary}\label{CorollaryCSM}
		Suppose that $M \in {\bf M}_n(\C)$ is complex symmetric and has distinct eigenvalues
		with corresponding unit eigenvectors $\vec{x}_1,\vec{x}_2,\ldots,\vec{x}_n$ satisfying $\inner{ \vec{x}_i,\vec{x}_j} \neq 0$
		for $i \neq j$.  The condition 
		\begin{equation}\label{eq-CSM}
			| \inner{ \vec{x}_i, \overline{\vec{x}_i} } |^2 = \prod_{\substack{j=1\\j\neq i}}^n (1 - | \inner{\vec{x}_j , \vec{x}_i }|^2)
		\end{equation}
		for $i=1,2,\ldots,n$ is necessary for $M$ to be unitarily equivalent to an analytic truncated
		Toeplitz operator.  If $n \leq 3$, then the preceding condition is also sufficient.
	\end{Corollary}
	
	\begin{proof}	
		Maintaining the notation and conventions of Theorem \ref{TheoremMain} and its proof, let
		\begin{equation*}
			\Theta(z) = \prod_{i=1}^n \frac{z - z_i}{1 - \overline{z_i}z } 
		\end{equation*}
		and observe that 	
		\begin{equation} \label{eq-product-formula}
			\prod_{\substack{j=1\\j\neq i}}^n (1 - | \inner{\vec{x}_i,\vec{x}_j}|^2)
			= \prod_{\substack{j=1\\j\neq i}}^n \rho^2(z_i,z_j)
			= (1 - |z_i|^2)^2|\Theta'(z_i)|^2
		\end{equation}
		holds by \eqref{eq-DistanceFormula}.
		Next note that the hypothesis upon the eigenvectors of $M$ implies that $M$ is irreducible.
		Since $M = M^t$, it follows that $M = JM^*J$ where $J$ denotes the canonical conjugation 
		\begin{equation*}
			J(\zeta_1,\zeta_2,\ldots,\zeta_n) = (\overline{\zeta_1} , \overline{ \zeta_2}, \ldots, \overline{ \zeta_n})
		\end{equation*}
		on $\C^n$.  If $J'$ is another conjugation
		which satisfies $M = J'M^*J'$, then $JJ'$ is unitary and commutes with $M$ whence $JJ'$ is a multiple of the identity.
		Thus $J$ is the unique conjugation, up to a unimodular constant factor, which satisfies $M = JM^tJ$.
		
		Now suppose that $UM = A_{\phi}^{\Theta}U$ for some unitary $U:\C^n \to \K_{\Theta}$.
		Since $A_{\phi}^{\Theta}$ is irreducible, a similar argument shows that 
		the conjugation $C:\K_{\Theta} \to \K_{\Theta}$ defined by $Cf = \overline{fz}\Theta$
		is the unique conjugation, up to a unimodular constant factor, such that
		$A_{\phi}^{\Theta} = C(A_{\phi}^{\Theta})^*C$ \cite{CCO,CSOA,Sarason}.  Since
		$A_{\phi}^{\Theta} = (UJU^*)(A_{\phi}^{\Theta})^*(UJU^*)$ and $UJU^*$ is a conjugation on $\K_{\Theta}$, it follows that
		$UJU^* = \gamma C$ for some unimodular constant $\gamma$.  This yields
		\begin{align*}
			| \inner{ \vec{x}_i,  \overline{\vec{x}_i}}|^2
			&= |\inner{\vec{x}_i,J\vec{x}_i}|^2 
			= | \inner{ U\vec{x}_i , UJ\vec{x}_i}|^2  \\
			&= | \inner{ U\vec{x}_i , CU\vec{x}_i}|^2 
			=| \inner{ \tilde{k}_{z_i}, k_{z_i} } |^2 \\
			&= (1 - |z_i|^2)^2 | \Theta'(z_i)|^2 \\
			&= \prod_{j=1}^n (1 - | \inner{ \vec{x}_j , \vec{x}_i } |^2)
		\end{align*}
		by \eqref{eq-product-formula}.
		
		To see that \eqref{eq-CSM} is sufficient in the $3 \times 3$ case, simply observe that 
		\eqref{eq-CSM} implies \eqref{eq-Choice1}, \eqref{eq-Choice2}, and \eqref{eq-Choice3}.
		In other words, \eqref{eq-CSM} implies that \eqref{eq-Crazy} holds whence $M$ is unitarily
		equivalent to an analytic truncated Toeplitz operator by Theorem \ref{Theorem3x3}.		
	\end{proof}

\section{Open problems}\label{SectionOpen}	
	Although there has been a recent surge in activity devoted to truncated Toeplitz operators
	under unitary equivalence \cite{CRW, TTOSIUES , STZ}, there are still many basic questions
	left unanswered.  We conclude this note with a series of open problems suggested by this work.
	
	\begin{Question}\label{QuestionMain}
		Is every complex symmetric matrix $M \in {\bf M}_n(\C)$ unitarily equivalent to a direct sum of truncated Toeplitz operators?
	\end{Question}	
	
	In other words, are truncated Toeplitz operators the basic building blocks of complex symmetric operators?	
	For $n \geq 4$
	numerical evidence strongly suggests that the implication (i) $\Leftrightarrow$ (ii)
	of Theorem \ref{Theorem3x3} fails generically.  In particular, \eqref{eq-CSM}
	tends to fail for all $i=1,2,\ldots,n$ for randomly generated complex symmetric matrices $M \in {\bf M}_n(\C)$ which satisfy the 
	hypotheses of Corollary \ref{CorollaryCSM}.  
	On the other hand, nothing that we know of prevents such a matrix from being unitarily equivalent to a truncated 
	Toeplitz operator with symbol in $L^{\infty}$ (as opposed to $H^{\infty}$). 

	\begin{Question}
		Let $n\geq 4$.  Is every complex symmetric matrix $M \in {\bf M}_n(\C)$
		having no pair of orthogonal, nonzero eigenvectors unitarily equivalent to a truncated Toeplitz operator?
	\end{Question}
	
	A variant of the preceding is:

	\begin{Question}
		Let $n\geq 4$.  Is every irreducible complex symmetric matrix $M \in {\bf M}_n(\C)$
		unitarily equivalent to a truncated Toeplitz operator?
	\end{Question}

	Recently, first author and J.~Tener \cite{UET} showed that every complex symmetric
	matrix is unitarily equivalent to a direct sum of (i) irreducible complex symmetric matrices 
	or (ii) matrices of the form $A \oplus A^t$ where $A$ is irreducible
	and not UECSM (such matrices are necessarily $6 \times 6$ or larger).  This immediately 
	suggests the following question.

	\begin{Question}\label{QuestionTranspose}
		For $A \in {\bf M}_n(\C)$, is 
		the matrix $A \oplus A^t \in {\bf M}_{2n}(\C)$ unitarily equivalent to a direct sum of truncated
		Toeplitz operators?
	\end{Question}
	
	Let $S$ denote the unilateral shift and recall that $S$ is not a complex symmetric operator
	\cite[Ex.~2.14]{CCO}, \cite[Prop.~1]{CSOA}, \cite[Thm.~4]{SNCSO}, \cite[Cor.~7]{MUCFO}.
	On the other hand, $T = S^* \oplus S$ is a complex symmetric operator \cite[Ex.~5]{G-P-II} which
	appears to be a promising candidate for a counterexample to Question \ref{QuestionTranspose}
	in the infinite-dimensional setting.

	One method for producing complex symmetric matrix representations of a given truncated Toeplitz
	operator is to use \emph{modified Aleksandrov-Clark bases} for $\K_{\Theta}$.  We refer the
	reader to \cite[Sect.~2.3, 5.2]{NLEPHS} for specific details.

	\begin{Question}
		Suppose that $M \in {\bf M}_n(\C)$ is complex symmetric.  If $M$
		is unitarily equivalent to a truncated Toeplitz operator, does there exist an inner function $\Theta$,
		a symbol $\phi \in L^{\infty}$, and a modified Aleksandrov-Clark basis $\beta$ for $\K_{\Theta}$
		such that $M$ is the matrix representation of $A_{\phi}^{\Theta}$ with respect to the basis $\beta$?
	\end{Question}

	In other words, do all such unitary equivalences between complex symmetric matrices and truncated
	Toeplitz operators arise essentially from Aleksandrov-Clark representations?

\bibliography{UEATTO}

\end{document}